\documentclass[12pt]{amsart}
\usepackage[all]{xy}
\usepackage[parfill]{parskip}

\usepackage{verbatim}
\usepackage{color}

\usepackage{amsmath, amscd, graphicx, latexsym, hyperref, times}
\usepackage{esint} 
\usepackage{graphicx}
\usepackage[abs]{overpic}
\usepackage{tikz}
\usepackage{tikz-cd}
\usetikzlibrary{arrows, patterns}

\oddsidemargin .25in \theoremstyle{plain}
\newtheorem{theorem}{Theorem}

\newtheorem{lemma}{Lemma}

\newtheorem{remark}{Remark}

\numberwithin{theorem}{section}
\numberwithin{equation}{section}
\numberwithin{lemma}{section}
\numberwithin{proposition}{section}
\numberwithin{corollary}{section}
\numberwithin{remark}{section}
\allowdisplaybreaks[3]

\newcommand{\diam}{\mathrm{diam}} 
\usepackage{cleveref}
\crefname{theorem}{Theorem}{Theorem}
\crefname{lemma}{Lemma}{Lemmas}
\crefname{proposition}{Proposition}{Propositions}
\crefname{question}{Question}{Questions}
\crefname{conjecture}{Conjecture}{Conjectures}
\crefname{equation}{}{}
\crefname{section}{Section}{Section}
\crefname{figure}{Figure}{Figure}
\crefname{theoremx}{Theorem}{Theorem}

\newcommand{\Ric}{\operatorname{Ric}}

\usepackage{appendix}
\begin{document}

\title[Obata under Bakry-Émery]{Obata-type equations under the Bakry-Émery Ricci curvature conditions}

\author{Mijia Lai}
\address{School of Mathematical Sciences, Shanghai Jiao Tong University, Shanghai 200240, China}
\email{laimijia@sjtu.edu.cn}

\author{Yiwei Liu}
\address{Shanghai Center for Mathematical Sciences, Fudan University, Shanghai 200438, China}
\email{liuyiwei@fudan.edu.cn}

\begin{abstract}
    In this paper, based on the warped product structures determined by the Obata-type equations with Robin boundary condition, we establish some rigidity results for compact manifolds with smooth boundary under appropriate Bakry-Émery Ricci curvature conditions and some other assumptions. 
\end{abstract}
\maketitle
\section{Introduction}

The classical Obata theorem \cite{Obata} is one of the fundamental rigidity results in
Riemannian geometry. It asserts that the existence of a nonconstant solution
to the equation
\begin{align} \label{eq:main}
    \nabla^2 f + f g =0
\end{align}
on a complete Riemannian manifold forces the manifold to be a standard sphere. 

A basic geometric feature of the Obata equation was already implicit in the classical theory of concircular scalar fields of Tashiro \cite{Tashiro}: the equation \eqref{eq:main}
locally determines a warped product structure around every regular point of $f$. Indeed one has a similar conclusion for a more general equation of the form $\nabla^2 f+\psi(f) g=0$, which we refer to as an Obata-type equation.

The local warped product structure, however, does not by itself determine the global geometry of the manifold. Critical sets of $f$ present various obstructions to a clearer global picture. To obtain a precise classification, one needs additional global information or curvature assumptions, usually in the format of Ricci curvature lower under. 

In the classical Obata theorem, {\bf completeness} is precisely the global
hypothesis that turns the local geometric information carried by the
equation into a global rigidity statement. The key is to restrict \eqref{eq:main} to a geodesic, which yields isolated global maximum and minimum points of $f$. This makes the regular level sets shrink to two points, implying a sphere in the underlying topology.

Obata-type equations arise frequently as the equality cases of sharp geometric and analytic inequalities. In particular, they appear in the rigidity analysis of eigenvalue estimates for the Laplacian and the Steklov problem, as well as in related questions in conformal geometry. A classical prototype is provided by the Lichnerowicz-Obata theorem.
If a closed \(n\)-dimensional Riemannian manifold satisfies
$    \Ric \ge (n-1)g,$
then the first nonzero eigenvalue of the Laplacian satisfies
$    \lambda_1\ge n.$  In the equality case, the Bochner formula and the corresponding equality
conditions imply that a first eigenfunction satisfies $\nabla^2 f+fg=0.$
Obata's rigidity theorem then shows that the manifold must be isometric to
the standard sphere. 

Such rigidity phenomenon have also been investigated for manifolds
with boundary. In fact, the Lichnerowicz eigenvalue estimate has natural counterparts in this setting, and the corresponding equality
cases lead directly to Obata equations with different boundary conditions.
Reilly \cite{Reilly} proved that if a compact Riemannian manifold
$(\Omega^n,g)$ satisfies
$ \Ric\ge (n-1)g$ and has mean convex boundary, then its first Dirichlet eigenvalue satisfies
$\lambda_1^D(\Omega)\ge n$. 
Moreover, equality holds precisely for the standard hemisphere. In the
equality case, a first Dirichlet eigenfunction satisfies
\[
    \nabla^2 f+fg=0
    \qquad\text{in }\Omega,
    \qquad
    f=0
    \qquad\text{on }\partial\Omega.
\]
Thus the rigidity part of Reilly's theorem can be viewed as a boundary
version of Obata's theorem with Dirichlet boundary condition.

The corresponding Neumann problem was studied by Escobar \cite{Escobar}
and Xia \cite{Xia}. Under the same positive Ricci curvature lower bound
and the stronger assumption that the boundary is convex, they established
the sharp estimate
$\lambda_1^N(\Omega)\ge n$
for the first nonzero Neumann eigenvalue, with equality again characterizing
the standard hemisphere. The equality case relies on the Obata equation subject to Neumann boundary condition.  Chen, Lai and Wang
\cite{ChenLaiWang} initiated a systematic study of the Obata equation with a nonzero constant Robin boundary condition. In contrast with the
Dirichlet and Neumann cases, the Robin problem exhibits a substantially
richer geometry: the existence of a solution does not in general force
the manifold to be a spherical domain, and additional warped product
models naturally arise. 

This development suggests a general point of view: the Obata-type equation
first determines a strong local, and often global, warped product structure,
while additional global geometric assumptions are needed to rigidify the
transverse factor and to determine the manifold completely. In particular, for problems arising from the characterization of
the equality cases of sharp geometric and analytic inequalities,
one usually has an additional lower Ricci curvature bound on the
ambient manifold together with a lower mean curvature bound on
the boundary. 

The purpose of
the present paper is to carry out this rigidity program under
Bakry-\'Emery Ricci curvature lower bounds \cite{BakryEmery}. Let $\phi$ be a smooth function
and let $m$ be an effective dimension. Recall that the $m$-Bakry-\'Emery
Ricci tensor is defined by
\[
    \Ric_{\phi,m}
    =
    \Ric+\nabla^2\phi
    -\frac{1}{m-\dim\Omega}\,d\phi\otimes d\phi,
\]
with the usual convention that $m=\dim\Omega$ is allowed only when $\phi$ is
constant. The weighted mean curvature of the boundary is
\[
    H_\phi=H-\frac{\partial\phi}{\partial\nu}.
\]

The passage from lower Ricci curvature bounds to lower
Bakry-\'Emery Ricci curvature bounds is natural from the viewpoint of smooth metric measure geometry. Indeed, once a weighted
measure $e^{-\phi}\,dV_g$ is introduced, the Bakry-\'Emery Ricci tensor plays the role of the
ordinary Ricci tensor in the corresponding Bochner formula and in many
comparison and eigenvalue estimates. In particular, a large part of the
classical theory under Ricci curvature lower bounds admits weighted
counterparts in which $\Ric$ is replaced by $\Ric_{\phi,m}$ and the
mean curvature of the boundary is replaced by the weighted mean curvature
$H_\phi$. Since the usual Riemannian setting is recovered when $\phi$ is
constant, it is therefore natural to ask whether the rigidity phenomena
associated with Obata-type equations persist in this broader weighted
framework. The results of the present paper show that this is indeed the
case: suitable lower Bakry-\'Emery Ricci curvature bounds, together with
the corresponding weighted boundary assumptions, are strong enough to
determine the global geometry of manifolds carrying solutions of the
Obata-type equations.

We shall treat three Obata-type equations uniformly. They correspond to the spherical, Euclidean,
and hyperbolic models respectively. Let $(\Omega^{n+1},g)$ be a smooth compact connected Riemannian manifold
with smooth boundary $\Sigma=\partial\Omega$, and let $\nu$ be the outward unit
normal along $\Sigma$. We first consider the Obata-type equations
\begin{equation*}
    \nabla^2f+\kappa f g=0,
\end{equation*}
where $\kappa=1,0,-1$. It should be pointed out that under Dirichlet boundary condition $f|_{\Sigma}=0$, the only solution is $f\equiv 0$ when $\kappa\leq0$; when $\kappa=1$, every nontrivial solution forces $\Omega$ to be isometric to the standard hemisphere (see \cite{Reilly}). We also point out that under Neumann boundary condition $f_{\nu}|_{\Sigma}=0$, $f$ is constant when $\kappa=0$ and $f\equiv0$ when $\kappa=-1$; when $\kappa=1$, every nontrivial solution also forces $\Omega$ to be isometric to the standard hemisphere (see \cite{Escobar} and \cite{Xia}). Therefore, in this paper, we mainly consider the Obata-type equations with Robin boundary condition, that is,
\begin{equation} \label{eq:three-obata}
\left\{
\begin{aligned}
  \nabla^2f+\kappa f g&=0 &&\text{in } \Omega,\\
f_\nu &= \lambda f &&\text{on } \Sigma,
\end{aligned}
\right.   
\end{equation}
where
\[
    (\kappa,\lambda)
    =
    (1,a),\qquad
    (0,b),\qquad
    (-1,c).
\]

We shall assume that 
\[a>0,\qquad b>0,\qquad c>1.\]
Thus in all three cases
$\mu:=\lambda^2+\kappa>0.$

The geometry determined by \eqref{eq:three-obata} has been investigated in
several recent works\cite{ChenLaiWang, LaiZhou}. In general, one only has a warped product structure over the central hypersurface
\[
    \Omega_0=\{x\in\Omega:f(x)=0\}.
\]
In the presence of isolated critical points, the geometry of the underlying manifold can in some cases be determined more precisely.  
If critical points are absent, several rigidity results have also been obtained by imposing lower curvature
bounds either on the boundary or on the ambient manifold. For example,
suitable lower bounds for $\Ric_\Sigma$ force the corresponding manifold to
be a geodesic ball in the sphere, Euclidean space, or hyperbolic space
\cite{LiuYang}. On the other hand, ambient Ricci curvature lower bounds
together with lower bounds for the mean curvature of the boundary lead to
similar rigidity results \cite{LiuYang,XiaXiong}. 

Our first main result is the following.

\begin{theorem}\label{thm:product-rigidity}
Let $(\Omega^{n+1},g)$ be a smooth compact connected Riemannian manifold
with smooth connected boundary $\Sigma$, and let
$f\in C^\infty(\Omega)$ be a nonconstant function. Assume that
\[
    (\Sigma,g|_\Sigma)
    \cong
    (M^k,g_M)\times(N^{n-k},g_N),
    \qquad k\ge2,
\]
where $M$ and $N$ are closed connected manifolds. Let
\[
    H_N=\operatorname{tr}_{TN}h
\]
denote the partial trace of the second fundamental form of $\Sigma$
along the $N$-directions. Then the following statements hold.

\begin{enumerate}

\item Let $f$ be a solution to   \eqref{eq:three-obata} with $(\kappa,\lambda)=(1, a)$ and 
 assume that $f$ has no critical points.
If there exist $\phi\in C^\infty(M)$ and
$m\in[k,\infty)$ such that
\[ \Ric^M_{\phi,m}
    \ge
    (m-1)(a^2+1)g_M,
\qquad 
    H_N\le-\frac{n-k}{a},
\]
then
\[
    (M,g_M)
    \cong
    \mathbb S^k\left(\frac1{\sqrt{a^2+1}}\right)
\]
and $\Omega$ is isometric to
\[
    (\left[0,\tan^{-1}\frac1a\right]\times\mathbb S^k\times N, \quad 
    g
    =
    dr^2+\sin^2r\,g_{\mathbb S^k}
    +\frac{a^2+1}{a^2}\cos^2r\,g_N).
\]

\item
Let $f$ be a solution to  \eqref{eq:three-obata} with $(\kappa,\lambda)=(0, b)$. If there exist $\phi\in C^\infty(M)$ and
$m\in[k,\infty)$ such that
\[  \Ric^M_{\phi,m}
    \ge
    (m-1)b^2g_M
\qquad 
   H_N\le0,\]
then
\[
    (M,g_M)
    \cong
    \mathbb S^k\left(\frac1b\right)
\]
and $\Omega$ is isometric to the product manifold
\[
    B^{k+1}\left(\frac1b\right)\times N.
\]

\item
Let $f$ be a solution to   \eqref{eq:three-obata} with $(\kappa,\lambda)=(-1, c)$.
 If there exist $\phi\in C^\infty(M)$ and
$m\in[k,\infty)$ such that
\[
    \Ric^M_{\phi,m}
    \ge
    (m-1)(c^2-1)g_M,
\qquad
   H_N\le\frac{n-k}{c},
\]
then
\[
    (M,g_M)
    \cong
    \mathbb S^k\left(\frac1{\sqrt{c^2-1}}\right)
\]
and $\Omega$ is isometric to
\[
    (\left[0,\tanh^{-1}\frac1c\right]\times\mathbb S^k\times N, \quad 
    g
    =
    dr^2+\sinh^2r\,g_{\mathbb S^k}
    +\frac{c^2-1}{c^2}\cosh^2r\,g_N).
\]

\end{enumerate}
\end{theorem}

The point of Theorem \ref{thm:product-rigidity} is that only a Bakry-\'Emery Ricci
curvature lower bound on the factor $M$ and a one-sided bound on
the partial mean curvature in the $N$-directions are required.
The boundary Hessian equation, after taking the trace along each
$N$-fiber and integrating over the closed manifold $N$, forces
$f|_\Sigma$ to be independent of the $N$-variable. Remarkably,
the one-sided partial mean curvature inequality is then
automatically saturated and yields the full model identity
\[
    h|_{TN}=-\frac{\kappa}{\lambda}g_N.
\]
The problem is therefore reduced to the weighted factor $M$, where
the maximal diameter theorem gives the round sphere rigidity.
The ambient manifold is finally recovered from the warped product
structure induced by the Obata equation.

Our second result concerns an intrinsic Bakry-\'Emery curvature condition on
the ambient manifold. 

\begin{theorem}\label{thm:ambient-rigidity}
Let $(\Omega^{n+1},g)$ be a smooth compact connected Riemannian manifold
with smooth boundary $\Sigma$, let $f\in C^\infty(\Omega)$ be nonconstant,
and let $\phi\in C^\infty(\Omega)$ and $m\in[n+1,\infty)$.

\begin{enumerate}
\item
Let $f$ be a solution to  \eqref{eq:three-obata} with $(\kappa,\lambda)=(1, a)$, and suppose
\[
    \Ric^\Omega_{\phi,m}\ge(m-1)g,
    \qquad
    H_\phi\ge(m-1)a,
\]
then $\Omega$ is isometric to the geodesic ball of radius
$\tan^{-1}\frac1a$
in the standard sphere $\mathbb S^{n+1}$.

\item
Let $f$ be a solution to   \eqref{eq:three-obata} with $(\kappa,\lambda)=(0, b)$
and suppose 
\[
    \Ric^\Omega_{\phi,m}\ge0,
    \qquad
    H_\phi\ge(m-1)b,
\]
then $\Omega$ is isometric to the Euclidean ball of radius $1/b$.

\item
Let $f$ be a solution to   \eqref{eq:three-obata} with $(\kappa,\lambda)=(-1, c)$
and suppose
\[
    \Ric^\Omega_{\phi,m}\ge-(m-1)g,
    \qquad
    H_\phi\ge(m-1)c,
\]
then $\Omega$ is isometric to the geodesic ball of radius
$    \tanh^{-1}\frac1c$
in the hyperbolic space $\mathbb H^{n+1}$.
\end{enumerate}
\end{theorem}

The proof of Theorem \ref{thm:ambient-rigidity} has a useful
dimension-reduction feature. The Obata equation first produces a totally
geodesic hypersurface
\[
    \Omega_0=\{f=0\}.
\]
The ambient weighted curvature condition then descends naturally to a
Bakry-\'Emery Ricci lower bound on $\Omega_0$, with the effective dimension
dropping from $m$ to $m-1$. At the same time, the weighted mean curvature
condition on $\Sigma$ induces the corresponding weighted mean curvature
bound on $\partial\Omega_0$. The Robin condition determines a natural
distance-type function on $\Omega_0$, whose level sets connect the maximal
level of the boundary graph function to $\partial\Omega_0$. This produces
the sharp lower bound for the inradius of $\Omega_0$. The weighted maximal
inscribed-radius rigidity theorem then shows that $\Omega_0$ is a geodesic
ball in the corresponding space form. Finally, the original Obata equation
allows one to reconstruct $\Omega$ and yields the desired rigidity.


The paper is organized as follows. In Section~2 we collect several basic facts on Obata-type equation and recall two important results in Bakry-\'Emery Ricci curvature lower bounds: the 
maximal diameter theorem and the maximal inscribed-radius theorem. In Section~3 we prove
Theorem~\ref{thm:product-rigidity}. In Section~4 we prove
Theorem~\ref{thm:ambient-rigidity}.

\section{Preliminaries}\label{sec:preliminaries}

In this section we fix our notation and collect several basic facts on Obata equations that
will be used throughout the paper. We also record several theorem that proofs rely on.

Let $(P^d,g)$ be a $d$-dimensional Riemannian manifold and
$\phi\in C^\infty(P)$. For $m\in[d,\infty)$, the $m$-Bakry-\'Emery
Ricci tensor is defined by
\[
    \Ric_{\phi,m}
    =
    \Ric+\nabla^2\phi
    -\frac{1}{m-d}\,d\phi\otimes d\phi,
\]
where, as usual, the case $m=d$ is allowed only when $\phi$ is constant.
We also write
\[
    \Ric_\phi=\Ric+\nabla^2\phi
\]
for the $\infty$-Bakry-\'Emery Ricci tensor.

If $(P,g)$ has smooth boundary and $\nu$ denotes the outward unit normal,
we use the convention
\[
    h(X,Y)=-\langle\nabla_XY,\nu\rangle
\]
for the second fundamental form and
\[
    H=\operatorname{tr}_g h
\]
for the mean curvature. The weighted mean curvature is
\[
    H_\phi=H-\langle\nabla\phi,\nu\rangle.
\]

For $\kappa\in\{1,0,-1\}$, set
\[
    s_\kappa(t)=
    \begin{cases}
        \sin t, & \kappa=1,\\
        t, & \kappa=0,\\
        \sinh t, & \kappa=-1,
    \end{cases}
\qquad
    c_\kappa(t)=s_\kappa'(t)=
    \begin{cases}
        \cos t, & \kappa=1,\\
        1, & \kappa=0,\\
        \cosh t, & \kappa=-1.
    \end{cases}
\]
Thus
\[
    s_\kappa''+\kappa s_\kappa=0,
    \qquad
    c_\kappa''+\kappa c_\kappa=0.
\]

The three equations considered in this paper can then be written uniformly
as
\begin{equation}\label{eq:unified-obata}
    \nabla^2f+\kappa f g=0
    \qquad\text{in }\Omega,
\end{equation}
together with
\begin{equation}\label{eq:unified-robin}
    f_\nu=\lambda f
    \qquad\text{on }\Sigma,
\end{equation}
where
\[
    (\kappa,\lambda)
    =
    (1,a),\qquad
    (0,b),\qquad
    (-1,c),
\]
respectively. Notice that in all three cases
\[
    \mu:=\lambda^2+\kappa>0.
\]

We begin with some elementary consequences of
\eqref{eq:unified-obata}--\eqref{eq:unified-robin}.

\begin{lemma}\label{lem:basic-identities}
Let $f$ be a nonconstant solution of
\eqref{eq:unified-obata}. Then
\[
    |\nabla f|^2+\kappa f^2
\]
is constant on $\Omega$. After multiplying $f$ by a suitable positive
constant, we may normalize
\begin{equation}\label{eq:normalization}
    |\nabla f|^2+\kappa f^2=1.
\end{equation}
If in addition $f$ satisfies \eqref{eq:unified-robin}, then on $\Sigma$
we have
\begin{equation}\label{eq:boundary-energy}
    |\nabla_\Sigma f|^2+\mu f^2=1,
    \qquad
    \mu=\lambda^2+\kappa,
\end{equation}
and
\begin{equation}\label{eq:boundary-hessian}
    \nabla_\Sigma^2f
    =
    -f\bigl(\kappa g_\Sigma+\lambda h\bigr).
\end{equation}
Moreover, the zero set
\[
    \Omega_0=\{x\in\Omega:f(x)=0\}
\]
is a smooth totally geodesic hypersurface whenever it is nonempty.
\end{lemma}

\begin{proof}
For any vector field $X$,
\[
\begin{aligned}
    X\bigl(|\nabla f|^2+\kappa f^2\bigr)
    &=
    2\nabla^2f(X,\nabla f)+2\kappa f\,X(f)\\
    &=0,
\end{aligned}
\]
which proves the first assertion.

On $\Sigma$,
\[
    |\nabla f|^2
    =
    |\nabla_\Sigma f|^2+f_\nu^2,
\]
and hence \eqref{eq:boundary-energy} follows immediately from
\eqref{eq:normalization} and $f_\nu=\lambda f$.

For $X,Y\in T\Sigma$, our convention for the second fundamental form gives
\[
    \nabla^2f(X,Y)
    =
    \nabla_\Sigma^2f(X,Y)+h(X,Y)f_\nu.
\]
Combining this identity with
\eqref{eq:unified-obata} and \eqref{eq:unified-robin} yields
\eqref{eq:boundary-hessian}.

Finally, on $\Omega_0$ we have $|\nabla f|=1$ by
\eqref{eq:normalization}, so $\Omega_0$ is a regular level set. Moreover,
\[
    \nabla^2f=0
    \qquad\text{on }\Omega_0.
\]
The second fundamental form of a regular level set is a multiple of the
restriction of $\nabla^2f$ to its tangent bundle, and hence $\Omega_0$ is
totally geodesic.
\end{proof}

A basic fact underlying our arguments is that the Obata-type equation
determines a warped product structure. We recall here the global form that
will be used later; see
\cite{ChenLaiWang,LiuYang,XiaXiong}.

Set
\[
    I_\kappa=
    \begin{cases}
        (-\frac{\pi}{2},\frac{\pi}{2}), & \kappa=1,\\
        \mathbb R, & \kappa=0,-1.
    \end{cases}
\]

\begin{theorem}\label{thm:warped-structure}
Let $(\Omega^{n+1},g)$ be a smooth compact connected Riemannian manifold
with smooth connected boundary $\Sigma$, and let $f$ be a nonconstant
solution of
\eqref{eq:unified-obata}--\eqref{eq:unified-robin}. In the case
$\kappa=1$, assume in addition that $f$ has no critical points.

After the normalization \eqref{eq:normalization}, $\Omega$ is isometric
to a $\mathbb Z_2$-symmetric domain in
\[
    \Omega_0\times I_\kappa,
    \qquad
    g=dt^2+c_\kappa^2(t)g_{\Omega_0},
\]
and
\[
    f=s_\kappa(t).
\]
More precisely, the boundary is given by the two graphs
\[
    t=\pm\Phi(x),
    \qquad x\in\Omega_0,
\]
where
\[
    \Phi\in C^\infty(\mathring{\Omega}_0)\cap C(\Omega_0),
    \qquad
    \Phi>0\ \text{in }\mathring{\Omega}_0,
    \qquad
    \Phi=0\ \text{on }\partial\Omega_0,
\]
and $\Phi$ satisfies
\begin{equation}\notag \label{eq:graph-equation}
    \frac{c_\kappa(\Phi)}
    {\sqrt{
        1+c_\kappa^{-2}(\Phi)
        |\nabla_{\Omega_0}\Phi|^2
    }}
    =
    \lambda s_\kappa(\Phi).
\end{equation}
\end{theorem}

Thus the Obata-type equation already determines the geometry in the
direction normal to $\Omega_0$. The remaining rigidity problem is to
determine the intrinsic geometry of $\Omega_0$ and the graph function
$\Phi$.

The following simple observation will be repeatedly used in conjunction
with the weighted maximal diameter theorem.

\begin{lemma}\label{lem:distance}
Let $(P,g)$ be a connected Riemannian manifold and let $u\in C^\infty(P)$
satisfy
\begin{equation}\label{eq:distance-energy}
    |\nabla u|^2+\mu u^2=1
\end{equation}
for some constant $\mu>0$. Suppose that there exist $p_-,p_+\in P$ such
that
\[
    u(p_-)
    =
    -\frac1{\sqrt\mu},
    \qquad
    u(p_+)
    =
    \frac1{\sqrt\mu}.
\]
Then $  d(p_-,p_+)\ge\frac{\pi}{\sqrt\mu}.$

If equality holds and $\gamma:[0,\pi/\sqrt\mu]\to P$ is a minimizing
unit-speed geodesic joining $p_-$ to $p_+$, then
\begin{equation}\label{eq:u-along-geodesic}
    u(\gamma(t))
    =
    -\frac1{\sqrt\mu}
    \cos(\sqrt\mu\,t),
\end{equation}
and, for $0<t<\pi/\sqrt\mu$,
\[
    \nabla u(\gamma(t))
    =
    |\nabla u|(\gamma(t))\,\gamma'(t).
\]
\end{lemma}

\begin{proof}
On the set where $|u|<1/\sqrt\mu$, consider
\[
    F
    =
    \frac1{\sqrt\mu}
    \arcsin(\sqrt\mu\,u).
\]
By \eqref{eq:distance-energy},
\[
    |\nabla F|
    =
    \frac{|\nabla u|}
    {\sqrt{1-\mu u^2}}
    =1.
\]
Integrating along any curve from $p_-$ to $p_+$, and taking limits at the
two endpoints, gives
\[
    d(p_-,p_+)
    \ge
    F(p_+)-F(p_-)
    =
    \frac{\pi}{\sqrt\mu}.
\]
If equality holds, equality must hold in the above estimate along every
minimizing geodesic from $p_-$ to $p_+$. Hence
$\nabla F=\gamma'$ along $\gamma$, which gives
\eqref{eq:u-along-geodesic} and the last assertion.
\end{proof}

We next recall two rigidity theorems for manifolds satisfying lower
Bakry-\'Emery Ricci curvature bounds.

The first is the weighted version of Cheng's maximal diameter theorem;
see \cite{MR3028781,Qian,Ruan}.

\begin{theorem}
\label{thm:weighted-diameter}
Let $(P^d,g)$ be a closed Riemannian manifold, let
$\phi\in C^\infty(P)$, and let $m\in[d,\infty)$. Assume that
\[
    \Ric_{\phi,m}
    \ge
    (m-1)K g
\]
for some $K>0$. Then
\[
    \diam(P)\le\frac{\pi}{\sqrt K}.
\]
If equality holds, then
\[
    (P,g)
    \cong
    \mathbb S^d\left(\frac1{\sqrt K}\right).
\]
Moreover, in the equality case
\[
    m=d
\]
and $\phi$ is constant.
\end{theorem}

We also recall the weighted maximal inscribed-radius theorem of Sakurai
\cite{Sakurai}. Let $(P^d,g)$ be a compact Riemannian manifold with smooth
boundary. Its inscribed radius is
\[
    \operatorname{InRad}(P)
    =
    \sup_{x\in P}d(x,\partial P).
\]

Let $\kappa,\lambda\in\mathbb R$. We say that $(\kappa,\lambda)$ satisfies
the ball condition if there exists a geodesic ball in the simply connected
space form of constant sectional curvature $\kappa$ whose boundary has
constant mean curvature $(d-1)\lambda$. We denote its radius by
$C_{\kappa,\lambda}$.

\begin{theorem}
\label{thm:weighted-inradius}
Let $(P^d,g)$ be a smooth compact connected Riemannian manifold with
smooth boundary, let $\phi\in C^\infty(P)$, and let
$m\in[d,\infty)$. Suppose that $(\kappa,\lambda)$ satisfies the ball
condition and that
\[
    \Ric_{\phi,m}
    \ge
    (m-1)\kappa g,
    \qquad
    H_\phi
    \ge
    (m-1)\lambda.
\]
Then
\[
    \operatorname{InRad}(P)
    \le
    C_{\kappa,\lambda}.
\]
If equality is attained, then $P$ is isometric to the geodesic ball of
radius $C_{\kappa,\lambda}$ in the simply connected space form of
constant curvature $\kappa$. Moreover, in the equality case
\[
    m=d
\]
and $\phi$ is constant.
\end{theorem}

For the three pairs occurring in this paper, one has
\[
    C_{1,a}
    =
    \tan^{-1}\frac1a,
    \qquad
    C_{0,b}
    =
    \frac1b,
    \qquad
    C_{-1,c}
    =
    \tanh^{-1}\frac1c.
\]

We conclude this section with a simple observation which will be useful in
the proof of the ambient rigidity theorem. It explains how the weighted
curvature and boundary mean curvature assumptions descend from $\Omega$
to the totally geodesic hypersurface $\Omega_0$.

\begin{lemma}\label{lem:dimension-reduction}
Let $(\Omega^{n+1},g)$ and $f$ be as in
Theorem~\ref{thm:warped-structure}, and let
\[
    \Omega_0=\{f=0\}.
\]
Suppose that, for some $m\in[n+1,\infty)$,
\begin{equation} \notag 
    \Ric^\Omega_{\phi,m}
    \ge
    (m-1)\kappa g \qquad
    H_\phi
    \ge
    (m-1)\lambda.
\end{equation}
Then, with $\phi$ also denoting its restriction to $\Omega_0$,
\begin{equation}\label{eq:reduced-ricci}
    \Ric^{\Omega_0}_{\phi,m-1}
    \ge
    (m-2)\kappa g_{\Omega_0},
\end{equation}
and
\begin{equation}\label{eq:reduced-mean}
    H_{\phi,\partial\Omega_0}
    \ge
    (m-2)\lambda.
\end{equation}
\end{lemma}

\begin{proof}
By Lemma~\ref{lem:basic-identities}, $\Omega_0$ is totally geodesic.
In the warped product coordinates of
Theorem~\ref{thm:warped-structure},
\[
    g=dt^2+c_\kappa^2(t)g_{\Omega_0},
\]
and, with our curvature convention,
\[
    R^\Omega(\partial_t,X,\partial_t,X)
    =
    -\kappa |X|^2
\]
at $t=0$, for every $X\in T\Omega_0$. Hence the Gauss equation gives
\[
    \Ric^{\Omega_0}(X,X)
    =
    \Ric^\Omega(X,X)-\kappa|X|^2.
\]
Since $\Omega_0$ is totally geodesic,
\[
    \nabla_{\Omega_0}^2\phi(X,X)
    =
    \nabla_\Omega^2\phi(X,X).
\]
Furthermore,
\[
    (m-1)-n
    =
    m-(n+1).
\]
Therefore
\[
\begin{aligned}
    \Ric^{\Omega_0}_{\phi,m-1}(X,X)
    &=
    \Ric^\Omega_{\phi,m}(X,X)
    -\kappa|X|^2\\
    &\ge
    (m-2)\kappa|X|^2,
\end{aligned}
\]
which proves \eqref{eq:reduced-ricci}.

It remains to prove the boundary estimate. Let
$P\in\partial\Omega_0$. Since
\[
    f(P)=0,
    \qquad
    f_\nu(P)=\lambda f(P)=0,
\]
the outward normal $\nu$ of $\Sigma$ is tangent to $\Omega_0$ and is
therefore also the outward normal of $\partial\Omega_0$ in $\Omega_0$.

Choose an orthonormal frame
$\{e_1,\ldots,e_n\}$ in $T_P\Sigma$ such that
\[
    e_1=\nabla f(P).
\]
Differentiating the Robin condition tangentially and using
$\nabla^2f(e_i,\nu)=0$ on $\Omega_0$, we obtain
\[
    h(e_1,e_1)=\lambda,
    \qquad
    h(e_1,e_i)=0,
    \quad i=2,\ldots,n.
\]
Since $\Omega_0$ is totally geodesic, the second fundamental form of
$\partial\Omega_0$ in $\Omega_0$ is the restriction of $h$ to
$T\partial\Omega_0$. Consequently,
\[
\begin{aligned}
    H_{\phi,\partial\Omega_0}
    &=
    H_\phi-\lambda\\
    &\ge
    (m-2)\lambda,
\end{aligned}
\]
which proves \eqref{eq:reduced-mean}.
\end{proof}

\section{Rigidity from a product structure on the boundary}
\label{sec:product-rigidity}

In this section we prove Theorem~\ref{thm:product-rigidity}. One of the
main points is that the one-sided bound on the partial mean curvature in
the $N$-directions already forces the full model second fundamental form
in those directions.

As before, we treat the three cases simultaneously. Set
\[
    (\kappa,\lambda)
    =
    (1,a),\qquad
    (0,b),\qquad
    (-1,c),
\]
and let
\[
    \mu=\lambda^2+\kappa>0,
    \qquad
    \ell=n-k=\dim N.
\]
Thus $f$ satisfies \eqref{eq:three-obata}.
The assumption on the partial mean curvature can be written uniformly as
\begin{equation}\label{eq:sec3-partial-mean}
    H_N
    \le
    -\frac{\ell\kappa}{\lambda}.
\end{equation}

By Lemma~\ref{lem:basic-identities}, we may assume \eqref{eq:normalization} holds.
On the boundary this becomes
\begin{equation}\label{eq:sec3-boundary-energy}
    |\nabla_\Sigma f|^2+\mu f^2=1,
\end{equation}
while the tangential Hessian satisfies
\begin{equation}\label{eq:sec3-boundary-hessian}
    \nabla_\Sigma^2f
    =
    -f(\kappa g_\Sigma+\lambda h).
\end{equation}

We first show that $f|_\Sigma$ is independent of the $N$-variable.
This is precisely where the one-sided assumption
\eqref{eq:sec3-partial-mean} enters the proof.

Fix $p\in M$ and define
\[
    u_p:N\longrightarrow\mathbb R,
    \qquad
    u_p(z)=f(p,z).
\]
Since
\[
    \Sigma=M\times N
\]
is a Riemannian product, the fiber $\{p\}\times N$ is totally geodesic
in $\Sigma$. Taking the trace of
\eqref{eq:sec3-boundary-hessian} in the $N$-directions gives
\begin{equation}\label{eq:sec3-fiber-equation}
    \Delta_Nu_p
    +
    \bigl(\ell\kappa+\lambda H_N\bigr)u_p
    =0.
\end{equation}
Multiplying \eqref{eq:sec3-fiber-equation} by $u_p$ and integrating over
the closed manifold $N$, we obtain
\begin{equation}\label{eq:sec3-fiber-integral}
    0
    =
    -\int_N|\nabla_Nu_p|^2
    +
    \int_N
    \bigl(\ell\kappa+\lambda H_N\bigr)u_p^2.
\end{equation}
By \eqref{eq:sec3-partial-mean},
\[
    \ell\kappa+\lambda H_N\le0.
\]
Therefore both terms on the right-hand side of
\eqref{eq:sec3-fiber-integral} are nonpositive. It follows that
\[
    \int_N|\nabla_Nu_p|^2=0,
\]
and hence $u_p$ is constant on $N$.

Since $p\in M$ was arbitrary, we conclude that
\begin{equation}\label{eq:sec3-f-independent-N}
    f(p,z)=\bar f(p)
\end{equation}
for some smooth function $\bar f$ on $M$. In particular,
\[
    \nabla_Nf\equiv0
    \qquad\text{on }\Sigma.
\]

This already implies that the inequality
\eqref{eq:sec3-partial-mean} is rigid. Indeed, for
$X,Y\in TN$, the product structure and
\eqref{eq:sec3-f-independent-N} give
\[
    \nabla_\Sigma^2f(X,Y)=0.
\]
Thus \eqref{eq:sec3-boundary-hessian} yields
\[
    f\bigl(
        \kappa g_N(X,Y)+\lambda h(X,Y)
    \bigr)
    =0.
\]
At every point where $f\neq0$, we therefore have
\[
    h(X,Y)
    =
    -\frac{\kappa}{\lambda}g_N(X,Y).
\]
On the other hand, the set $\{f\neq0\}\cap\Sigma$ is dense in $\Sigma$.
Indeed, if $f$ vanished on an open subset of $\Sigma$, then
$\nabla_\Sigma f$ would also vanish there, contradicting
\eqref{eq:sec3-boundary-energy}. By continuity we conclude that
\begin{equation}\label{eq:sec3-h-equality}
    h|_{TN}
    =
    -\frac{\kappa}{\lambda}g_N
    \qquad\text{on }\Sigma.
\end{equation}

Thus the one-sided partial mean curvature condition automatically
recovers the model second fundamental form in all $N$-directions.

Since $f$ is independent of $N$, 
\eqref{eq:sec3-boundary-energy} reduces to
\begin{equation}\label{eq:sec3-M-energy}
    |\nabla_M\bar f|^2+\mu\bar f^2=1.
\end{equation}
The function $\bar f$ is nonconstant. Let $p_+$ and $p_-$ be points of
$M$ at which $\bar f$ attains its maximum and minimum, respectively.
At these two points $\nabla_M\bar f=0$, and
\eqref{eq:sec3-M-energy} implies
\[
    |\bar f(p_\pm)|
    =
    \frac1{\sqrt\mu}.
\]
Since $\bar f$ is nonconstant, its maximum and minimum cannot have the
same sign. After replacing $f$ by $-f$ if necessary, we may therefore
assume that
\[
    \bar f(p_+)=\frac1{\sqrt\mu},
    \qquad
    \bar f(p_-)=-\frac1{\sqrt\mu}.
\]

Applying Lemma~\ref{lem:distance} to
\eqref{eq:sec3-M-energy}, we obtain
\begin{equation}\label{eq:sec3-distance-lower}
    d_M(p_-,p_+)
    \ge
    \frac{\pi}{\sqrt\mu}.
\end{equation}

On the other hand, the Bakry-\'Emery Ricci curvature assumption in
Theorem~\ref{thm:product-rigidity} is precisely
\[
    \Ric^M_{\phi,m}
    \ge
    (m-1)\mu g_M.
\]
By the weighted maximal diameter theorem,
\[
    \diam(M)
    \le
    \frac{\pi}{\sqrt\mu}.
\]
Together with \eqref{eq:sec3-distance-lower}, this gives
\[
    d_M(p_-,p_+)
    =
    \diam(M)
    =
    \frac{\pi}{\sqrt\mu}.
\]
Hence the equality case of
Theorem~\ref{thm:weighted-diameter} yields
\[
    (M,g_M)
    \cong
    \mathbb S^k\left(\frac1{\sqrt\mu}\right).
\]

Moreover, equality holds in the distance estimate. Therefore, along
every unit-speed minimizing geodesic $\gamma$ from $p_-$ to $p_+$,
\[
    \bar f(\gamma(r))
    =
    -\frac1{\sqrt\mu}\cos(\sqrt\mu\,r).
\]
Using polar coordinates centered at $p_+$, we obtain globally
\[
    \bar f(p)
    =
    \frac1{\sqrt\mu}
    \cos\bigl(\sqrt\mu\,d_M(p,p_+)\bigr).
\]
Consequently,
\[
    \Sigma
    =
    \mathbb S^k\left(\frac1{\sqrt\mu}\right)\times N,
\]
and the restriction of $f$ to $\Sigma$ depends only on the spherical
factor.

Let
\[
    \Sigma_+=\{x\in\Sigma:f(x)\ge0\}.
\]
Using polar coordinates centered at $p_+$ on $M$, we may write
\[
    \Sigma_+
    =
    \mathbb S^{k-1}
    \times
    \left[0,\frac{\pi}{2\sqrt\mu}\right]_r
    \times N,
\]
with metric
\begin{equation}\label{eq:sec3-Sigma-plus}
    g_{\Sigma_+}
    =
    dr^2
    +
    \frac1\mu\sin^2(\sqrt\mu\,r)
    g_{\mathbb S^{k-1}}
    +
    g_N,
\end{equation}
and
\[
    f
    =
    \frac1{\sqrt\mu}\cos(\sqrt\mu\,r).
\]

By Theorem~\ref{thm:warped-structure}, $\Omega$ is a
$\mathbb Z_2$-symmetric domain in
\[
    \Omega_0\times I_\kappa,
    \qquad
    g
    =
    dt^2+c_\kappa^2(t)g_{\Omega_0},
\]
and
\[
    f=s_\kappa(t).
\]
The positive part of the boundary is the graph
\[
    \Sigma_+
    =
    \{(x,\Phi(x)):x\in\Omega_0\}.
\]
Since $f|_{\Sigma_+}$ depends only on $r$,
the graph function also depends only on $r$:
\[
    \Phi=\Phi(r).
\]
Furthermore,
\begin{equation}\label{eq:sec3-Phi-r}
    s_\kappa(\Phi(r))
    =
    \frac1{\sqrt\mu}\cos(\sqrt\mu\,r).
\end{equation}

For
\[
    \tau_\kappa(s)
    =
    \frac{s_\kappa(s)}{c_\kappa(s)},
\]
introduce a new radial coordinate $s$ by
\begin{equation}\label{eq:sec3-coordinate-change}
    \tau_\kappa(s)
    =
    \frac1\lambda\sin(\sqrt\mu\,r).
\end{equation}
Let
\[
    R=C_{\kappa,\lambda}.
\]
Then
\[
    \tau_\kappa(R)=\frac1\lambda,
    \qquad
    s_\kappa(R)=\frac1{\sqrt\mu},
\]
so that $s\in[0,R]$.

Comparing the metric induced on the graph $t=\Phi(r)$ with
\eqref{eq:sec3-Sigma-plus}, and using
\eqref{eq:sec3-Phi-r} and
\eqref{eq:sec3-coordinate-change}, we obtain
\[
    g_{\Omega_0}
    =
    ds^2
    +
    s_\kappa^2(s)g_{\mathbb S^{k-1}}
    +
    \frac{\mu}{\lambda^2}
    c_\kappa^2(s)g_N.
\]
Thus
\[
    \Omega_0
    \cong
    [0,R]\times\mathbb S^{k-1}\times N.
\]

Using
\eqref{eq:sec3-Phi-r} and
\eqref{eq:sec3-coordinate-change}, the boundary equation can be written
in the form
\[
    \frac{\tau_\kappa^2(s)}{\tau_\kappa^2(R)}
    +
    \frac{s_\kappa^2(t)}{s_\kappa^2(R)}
    =1.
\]
Hence
\[
    \Omega
    =
    B_\kappa(R)\times N,
\]
where
\[
    B_\kappa(R)
    =
    \left\{
    (\omega,s,t):
    \frac{\tau_\kappa^2(s)}{\tau_\kappa^2(R)}
    +
    \frac{s_\kappa^2(t)}{s_\kappa^2(R)}
    \le1
    \right\},
\]
and the metric is
\begin{equation}\label{eq:sec3-product-before-polar}
\begin{split}
    g
    ={}&
    dt^2
    +
    c_\kappa^2(t)
    \bigl(
        ds^2+s_\kappa^2(s)g_{\mathbb S^{k-1}}
    \bigr)\\
    &+
    \frac{\mu}{\lambda^2}
    c_\kappa^2(t)c_\kappa^2(s)g_N.
\end{split}
\end{equation}

The first line of \eqref{eq:sec3-product-before-polar}, restricted to
$B_\kappa(R)$, is the metric of the geodesic ball of radius $R$ in the
$(k+1)$-dimensional simply connected space form of sectional curvature
$\kappa$. Under the corresponding geodesic polar coordinates
$(\rho,\theta)\in[0,R]\times\mathbb S^k$, one has
\[
    c_\kappa(t)c_\kappa(s)=c_\kappa(\rho).
\]
Therefore \eqref{eq:sec3-product-before-polar} becomes
\begin{equation}\label{eq:sec3-final}
    g
    =
    d\rho^2
    +
    s_\kappa^2(\rho)g_{\mathbb S^k}
    +
    \frac{\mu}{\lambda^2}
    c_\kappa^2(\rho)g_N,
    \qquad
    0\le\rho\le R.
\end{equation}

Finally, the three choices of $(\kappa,\lambda)$ give
\[
\begin{array}{c|c|c|c}
\kappa & \lambda & R=C_{\kappa,\lambda}
& \mu=\lambda^2+\kappa\\ \hline
1 & a & \tan^{-1}(1/a) & a^2+1\\[1mm]
0 & b & 1/b & b^2\\[1mm]
-1 & c & \tanh^{-1}(1/c) & c^2-1.
\end{array}
\]
Substituting these values into \eqref{eq:sec3-final} gives precisely
the three metrics stated in Theorem~\ref{thm:product-rigidity}.
\qed

\begin{remark}
Based on the weighted version of Cheng’s maximal diameter theorem, we also conclude that
\[
    m=k
    \qquad\text{and}\qquad
    \phi\equiv\mathrm{constant}.
\]
\end{remark}

\section{Rigidity under ambient Bakry-\'Emery curvature bounds}
\label{sec:ambient-rigidity}

We now prove Theorem~\ref{thm:ambient-rigidity}. As in the previous
section, the three cases will be treated simultaneously. Recall 
\[
    (\kappa,\lambda)
    =
    (1,a),\qquad(0,b),\qquad(-1,c),
\]
and set
\[
    \mu=\lambda^2+\kappa,
    \qquad
    R=C_{\kappa,\lambda}.
\]
Thus
\begin{equation}\label{eq:sec4-model-radius}
    \tau_\kappa(R)=\frac1\lambda,
    \qquad
    s_\kappa(R)=\frac1{\sqrt\mu}.
\end{equation}

The assumptions of the theorem are
\begin{equation}\notag
    \Ric^\Omega_{\phi,m}
    \ge
    (m-1)\kappa g
\qquad
    H_\phi
    \ge
    (m-1)\lambda.
\end{equation}

We first justify the use of
Theorem~\ref{thm:warped-structure}. Only the spherical case requires an
additional argument, since there we must exclude critical points of $f$.

Suppose that $\kappa=1$ and that $p\in\Omega^\circ$ is a critical point.
After normalization,
\[
    |\nabla f|^2+f^2=1,
\]
so $f(p)=\pm1$. Let
$\gamma:[0,\ell]\to\Omega$ be a minimizing geodesic from $p$ to
$\Sigma$. The weighted inscribed-radius theorem applied directly to
$\Omega$ gives
\[
    \ell
    \le
    R
    =
    \tan^{-1}\frac1a
    <
    \frac{\pi}{2}.
\]
Along $\gamma$,
\[
    (f\circ\gamma)''+(f\circ\gamma)=0,
    \qquad
    (f\circ\gamma)'(0)=0,
\]
and hence
\[
    f(\gamma(t))
    =
    \pm\cos t.
\]
Since $\gamma'(\ell)=\nu$, the Robin condition would imply
\[
    \mp\sin\ell
    =
    \pm a\cos\ell,
\]
which is impossible for $0<\ell<\pi/2$. Thus $f$ has no critical points.

We may therefore apply
Theorem~\ref{thm:warped-structure} in all three cases. We obtain
\[
    \Omega
    \subset
    \Omega_0\times I_\kappa,
    \qquad
    g=dt^2+c_\kappa^2(t)g_{\Omega_0},
\]
where
\[
    \Omega_0=\{f=0\},
    \qquad
    f=s_\kappa(t),
\]
and the boundary is represented by the graphs
\[
    t=\pm\Phi(x),
    \qquad x\in\Omega_0.
\]
The graph function satisfies
\begin{equation}\label{eq:sec4-graph}
    \frac{c_\kappa(\Phi)}
    {\sqrt{
       1+c_\kappa^{-2}(\Phi)
       |\nabla_{\Omega_0}\Phi|^2
    }}
    =
    \lambda s_\kappa(\Phi).
\end{equation}

Since the left-hand side of \eqref{eq:sec4-graph} is at most
$c_\kappa(\Phi)$,
\[
    \lambda s_\kappa(\Phi)
    \le
    c_\kappa(\Phi).
\]
Hence
\[
    0\le\Phi\le R.
\]
Because $\Phi>0$ in the interior and $\Phi=0$ on
$\partial\Omega_0$, it attains a positive maximum at an interior point.
At a maximum point $\nabla\Phi=0$, so
\eqref{eq:sec4-graph} gives
\[
    c_\kappa(\Phi)=\lambda s_\kappa(\Phi),
\]
and consequently
\[
    \max_{\Omega_0}\Phi=R.
\]
Set
\[
    A=\{x\in\Omega_0:\Phi(x)=R\}.
\]
Then $A$ is nonempty and compact.

Define $v:\Omega_0\to[0,R]$ implicitly by
\begin{equation}\label{eq:sec4-v-definition}
    \tau_\kappa(v)
    =
    \frac1\lambda
    \sqrt{1-\mu s_\kappa^2(\Phi)}.
\end{equation}
By \eqref{eq:sec4-model-radius},
\[
    v=0
    \quad\text{on }A,
    \qquad
    v=R
    \quad\text{on }\partial\Omega_0.
\]
Moreover, $v$ is smooth wherever $0<v<R$.

We claim that
\begin{equation}\label{eq:sec4-eikonal}
    |\nabla_{\Omega_0}v|=1
    \qquad
    \text{on }
    \Omega_0\setminus(A\cup\partial\Omega_0).
\end{equation}
Indeed, squaring \eqref{eq:sec4-graph} gives
\[
    |\nabla\Phi|^2
    =
    \frac{
        c_\kappa^2(\Phi)
        \bigl(1-\mu s_\kappa^2(\Phi)\bigr)}
        {\lambda^2s_\kappa^2(\Phi)}.
\]
Differentiating \eqref{eq:sec4-v-definition} and using
\[
    \tau_\kappa'(t)=\frac1{c_\kappa^2(t)}
\]
then gives \eqref{eq:sec4-eikonal} by a direct calculation.

For $t\in[0,R]$, set
\[
    T_t=\{x\in\Omega_0:v(x)=t\}.
\]
Let $x\in T_t$ and let $\gamma$ be any curve joining a point of $A$
to $x$. Using \eqref{eq:sec4-eikonal},
\[
    t
    =
    v(x)-v(A)
    \le
    L(\gamma).
\]
Taking the infimum over all such curves yields
\begin{equation}\label{eq:sec4-distance-A}
    d(A,T_t)\ge t.
\end{equation}
In particular,
\begin{equation}\label{eq:sec4-inradius-lower}
    d(A,\partial\Omega_0)\ge R.
\end{equation}

We now apply the dimension-reduction lemma.
By Lemma~\ref{lem:dimension-reduction},
\[]
    \Ric^{\Omega_0}_{\phi,m-1}
    \ge
    (m-2)\kappa g_{\Omega_0}\qquad
    H_{\phi,\partial\Omega_0}
    \ge
    (m-2)\lambda.
\]

The weighted maximal inscribed-radius theorem, now applied to
$\Omega_0$, gives
\[
    \operatorname{InRad}(\Omega_0)
    \le
    R.
\]
Together with \eqref{eq:sec4-inradius-lower}, this shows that equality
holds. Hence
$(\Omega_0,g_{\Omega_0})$
is isometric to the geodesic ball of radius $R$ in the
$n$-dimensional simply connected space form of sectional curvature
$\kappa$.

Moreover, $A$ consists of a single point, say $x_0$, which is the center
of this ball.

We next identify $v$. From \eqref{eq:sec4-distance-A},
\[
    v(x)\le d(x_0,x).
\]
Applying the same argument from $x$ to the boundary gives
\[
    R-v(x)
    \le
    d(x,\partial\Omega_0).
\]
Since $\Omega_0$ is the model ball centered at $x_0$,
\[
    d(x,\partial\Omega_0)
    =
    R-d(x_0,x).
\]
Consequently,
\[
    v(x)\ge d(x_0,x),
\]
and therefore
\begin{equation}\label{eq:sec4-v-distance}
    v(x)=d(x_0,x).
\end{equation}

Substituting \eqref{eq:sec4-v-distance} into
\eqref{eq:sec4-v-definition}, we obtain
\begin{equation}\label{eq:sec4-boundary-equation}
    \frac{
        \tau_\kappa^2(d(x,x_0))
    }{
        \tau_\kappa^2(R)
    }
    +
    \frac{
        s_\kappa^2(\Phi(x))
    }{
        s_\kappa^2(R)
    }
    =1.
\end{equation}

The equation \eqref{eq:sec4-boundary-equation} is exactly the equation
of the geodesic sphere of radius $R$ in the $(n+1)$-dimensional simply
connected space form of curvature $\kappa$, written in the warped product
coordinates
\[
    dt^2+c_\kappa^2(t)g_{\Omega_0}.
\]
Equivalently, this is the same model calculation as in
\cite[Proposition~2.6]{LiuYang}. Therefore $\Sigma$ is the geodesic
sphere of radius $R$, and $\Omega$ is the corresponding geodesic ball.

Thus
\[
    (\Omega,g)
    \cong
    B_{\mathbb M_\kappa^{n+1}}(R),
\]
where $\mathbb M_\kappa^{n+1}$ denotes the simply connected
$(n+1)$-dimensional space form of sectional curvature $\kappa$.

Finally,
\[
    R=
    \begin{cases}
        \displaystyle \tan^{-1}\frac1a,
            & \kappa=1,\\[2mm]
        \displaystyle \frac1b,
            & \kappa=0,\\[2mm]
        \displaystyle \tanh^{-1}\frac1c,
            & \kappa=-1,
    \end{cases}
\]
which gives precisely the three conclusions of
Theorem~\ref{thm:ambient-rigidity}.
\qed

\begin{remark}
Once the above rigidity has been established, the weighted maximal
inscribed-radius theorem can be applied once more to $\Omega$ itself.
Since its inscribed radius is exactly $R=C_{\kappa,\lambda}$, the equality
case also implies
\[
    m=n+1
    \qquad\text{and}\qquad
    \phi\equiv\mathrm{constant}.
\]
Thus the weighted rigidity assumptions ultimately collapse to the
unweighted model in the equality case.
\end{remark}

\bibliographystyle{plain}
\bibliography{ref}

\vskip .5cm
\noindent

\end{document}